\documentclass[11pt,a4paper,leqno]{amsart}

\title
[Propagation of anisotropic Gabor wave front sets]
{Propagation of anisotropic Gabor wave front sets}

\author[P. Wahlberg]{Patrik Wahlberg}

\address{Dipartimento di Scienze Matematiche, Politecnico di Torino, Corso Duca degli Abruzzi 24,
10129 Torino, Italy}

\email{patrik.wahlberg[AT]polito.it}

\usepackage{latexsym}
\usepackage[a4paper]{geometry}
\usepackage{amsmath}
\usepackage{amssymb}
\usepackage{amsthm}
\usepackage{bbm}
\usepackage{amsfonts}
\usepackage{mathrsfs}
\usepackage{calc}
\usepackage{cite}
\usepackage{color}
\usepackage{graphicx}
\usepackage{enumerate}

\numberwithin{equation}{section}          

\newtheorem{thm}{Theorem}
\numberwithin{thm}{section}

\newcommand{\rubrik}{}
\newtheorem{prop}[thm]{Proposition}

\newtheorem{lem}[thm]{Lemma}

\theoremstyle{definition}

\newtheorem{defn}[thm]{Definition}

\theoremstyle{remark}

\newtheorem{rem}[thm]{Remark}              

\newcommand{\scal}[2]{\langle #1,#2\rangle}

\newcommand{\pdd}[2] {\partial_{#1} ^{#2}}

\newcommand{\ro}{\mathbf R}
\newcommand{\no}{\mathbf N}
\newcommand{\rr}[1]{\mathbf R^{#1}}
\newcommand{\sr}[1]{\mathbf S^{#1}}
\newcommand{\sro}[1]{\mathbf S}
\newcommand{\nn}[1]{\mathbf N^{#1}}

\newcommand{\dd}{\mathrm {d}}

\newcommand{\ep}{\varepsilon}
\newcommand{\fy}{\varphi}

\newcommand{\cdo}{\, \cdot \, }

\newcommand{\supp}{\operatorname{supp}}

\newcommand{\wpr}{{\text{\footnotesize $\#$}}}

\newcommand{\eabs}[1]{\langle #1\rangle}

\newcommand{\GL}{\operatorname{GL}}

\newcommand{\charac}{\operatorname{char}}

\newcommand{\csupp}{\operatorname{conesupp}}
\newcommand{\rB}{\operatorname{B}}

\newcommand{\WF}{\mathrm{WF}}
\newcommand{\WFg}{\mathrm{WF_{\rm g}}}
\newcommand{\WFgs}{\mathrm{WF_{g}^{\it s}}}

\newcommand{\cS}{\mathscr{S}}

\newcommand{\cF}{\mathscr{F}}

\newcommand{\cK}{\mathscr{K}}

\newcommand{\J}{\mathcal{J}}
\newcommand{\wt}{\widetilde}
\newcommand{\wh}{\widehat}

\def\la{\langle}
\def\ra{\rangle}
\newcommand{\leqs}{\leqslant}
\newcommand{\geqs}{\geqslant}

\begin{document}

\begin{abstract}
We show a result on propagation of the anisotropic Gabor wave front set
for linear operators with a tempered distribution Schwartz kernel. 
The anisotropic Gabor wave front set is parametrized by a positive parameter
relating the space and frequency variables. 
The anisotropic Gabor wave front set of the Schwartz kernel is assumed to satisfy a graph type criterion. 
The result is applied to a class of evolution equations that generalizes the Schr\"odinger equation 
for the free particle. 
The Laplacian is replaced by any partial differential operator with constant coefficients, 
real symbol and order at least two. 
\end{abstract}

\keywords{Tempered distributions, global wave front sets, microlocal analysis, phase space, anisotropy, propagation of singularities, evolution equations}
\subjclass[2010]{46F05, 46F12, 35A27, 47G30, 35S05, 35A18, 81S30, 58J47, 47D06}

\maketitle

\section{Introduction}\label{sec:intro}

The paper treats the propagation of the anisotropic Gabor wave front set 
for a class of continuous linear operators. 

H\"ormander \cite{Hormander1} introduced in 1991 the Gabor wave front set of a tempered distribution as a closed 
conic subset of the phase space $T^* \rr d \setminus 0$. 
It consists of directions in $T^* \rr d \setminus 0$ of global singularities, in no neighborhood of which the short-time Fourier transform decays superpolynomially. 
The Gabor wave front set is empty precisely when the tempered distribution is a Schwartz function, 
so it records smoothness and decay at infinity simultaneously.

Recent works \cite{Carypis1,Cordero1,PRW1,Rodino2,Schulz1,Wahlberg1,Wahlberg2} treat the Gabor wave front set and similar concepts.
The Gabor wave front set is identical to Nakamura's homogeneous wave front set \cite{Nakamura1,Schulz1}. 
H\"ormander's original paper \cite{Hormander1} contains results on the action of a linear continuous operator on the Gabor wave front set. 
Propagation of the Gabor wave front set for the solution to evolution equations with quadratic Hamiltonian with non-negative real part is treated in \cite{PRW1,Wahlberg2}. 
The singular space of such a quadratic form, introduced by Hitrik and Pravda--Starov \cite{Hitrik1}, 
then plays a crucial role. 

We have defined and studied an anisotropic version of the Gabor wave front set, which is parametrized by $s > 0$, in \cite{Rodino4}. 
The new feature is to replace the superpolynomial decay along straight lines in phase space $T^* \rr d \setminus 0$, characteristic to the Gabor wave front set, by decay along curves of the form 
\begin{equation*}
\ro_+ \ni \lambda \mapsto ( \lambda x, \lambda^s \xi) \in T^* \rr d \setminus 0
\end{equation*}
where $(x,\xi) \in T^* \rr d \setminus 0$. 
The resulting wave front set is baptized to the anisotropic $s$-Gabor wave front set, and 
it is denoted $\WF_{\rm g}^s (u) \subseteq T^* \rr d \setminus 0$ for a tempered distribution $u \in \cS'(\rr d)$. 
If $s = 1$ we recover the standard Gabor wave front set. 

In \cite{Rodino4} we develop pseudodifferential calculus and microlocal analysis for the anisotropic $s$-Gabor wave front set, inspired by e.g. \cite{Cappiello0,Cappiello4} which treat anisotropic partial differential operators with polynomial coefficients. 
This means that we study pseudodifferential calculus with symbol classes that are anisotropic modifications of the standard 
Shubin symbols. The anisotropic symbols \cite{Shubin1} satisfy estimates of the form
\begin{equation*}
|\pdd x \alpha \pdd \xi \beta a(x,\xi)|
\lesssim ( 1 + |x| + |\xi|^{\frac1s} )^{m - |\alpha| - s |\beta|}, \quad (x,\xi) \in T^* \rr d, \quad \alpha, \beta \in \nn d. 
\end{equation*}
It also means results on microlocality and microellipticity in the anisotropic framework. 

For this purpose we benefit from ideas and techniques from papers on microlocal analysis that is anisotropic in the dual (frequency) 
variables only (see e.g. \cite{Parenti1}), as opposed to our anisotropy which refers to the space and frequency variables comprehensively. 
An overall summary of \cite{Rodino4} is an anisotropic version of Shubin's calculus of pseudodifferential operators \cite{Shubin1}. 

The anisotropic $s$-Gabor wave front describes accurately the global singularities of oscillatory functions of 
chirp type \cite[Theorem~7.1]{Rodino4}. These are exponentials with real polynomial phase functions. 

In this paper the chief result concerns propagation of the anisotropic $s$-Gabor wave front set by
a continuous linear operator $\cK: \cS(\rr d) \to \cS'(\rr d)$ defined by a Schwartz kernel $K \in \cS'(\rr {2d})$. 
Suppose that the $s$-Gabor wave front set of $K$ contains no points of the form $(x, 0, \xi, 0)  \in T^* \rr {2d} \setminus 0$
nor of the form $(0, y, 0, -\eta) \in T^* \rr {2d} \setminus 0$, with $x, y, \xi, \eta \in \rr d$. 
(Roughly speaking this amounts to that $\WF_{\rm g}^s (K)$ resembles the graph of an invertible matrix.)
Then $\cK: \cS(\rr d) \to \cS(\rr d)$ acts continuously, extends uniquely to a sequentially continuous linear operator $\cK: \cS'(\rr d) \to \cS'(\rr d)$, 
and for $u \in \cS'(\rr d)$ we have
\begin{equation}\label{eq:linearpropagation}
\WF_{\rm g}^s (\cK u) \subseteq \WF_{\rm g}^s (K)' \circ \WF_{\rm g}^s (u).  
\end{equation}
Here we use the notation 
\begin{equation*}
A' \circ B  = \{ (x,\xi) \in \rr {2d}: \,  \exists (y,\eta) \in B: \, (x,y,\xi,-\eta) \in A \}
\end{equation*}
for $A \subseteq \rr {4d}$ and $B \subseteq \rr {2d}$. 

The inclusion \eqref{eq:linearpropagation} is conceptually similar to propagation results for other 
types of wave front sets, local \cite{Hormander0}, or global \cite{Carypis1,PRW1,Wahlberg1}. 

As an application of the inclusion \eqref{eq:linearpropagation} we study propagation of the anisotropic $s$-Gabor wave front set
for the initial value Cauchy problem for an evolution equation of the form 
\begin{equation*}
\left\{
\begin{array}{rl}
\partial_t u(t,x) + i p(D_x) u (t,x) & = 0, \quad x \in \rr d, \\
u(0,\cdot) & = u_0 \in \cS'(\rr d)
\end{array}
\right.
\end{equation*}
where $p: \rr d \to \ro$ is a polynomial with real coefficients of order $m \geqs 2$. 
This generalizes the Schr\"odinger equation for the free particle where $m = 2$ and $p(\xi) = |\xi|^2$. 

Provided $s = \frac{1}{m-1}$ we show that $\WF_{\rm g}^s$ of the solution at time $t \in \ro$ equals $\WF_{\rm g}^s (u_0)$ 
transported by the Hamilton flow $\chi_t $ with respect to the principal part $p_m $ of $p(\xi)$, that is 
\begin{equation*}
( x (t),\xi(t) ) = \chi_t (x, \xi)
= (x + t \nabla p_m (\xi), \xi), \quad t \in \ro, \quad ( x, \xi ) \in T^* \rr d \setminus 0. 
\end{equation*}

The conclusion is again conceptually similar to other results on propagation of singularities 
\cite{Hormander0,Carypis1,Wahlberg1}, and generalizes known results when $p$ is a homogeneous 
quadratic form and $s = 1$ \cite{PRW1}. 

The article \cite{Wahlberg3} contains results similar to those of this paper, but in the functional framework 
of Gelfand--Shilov spaces and their ultradistribution dual spaces. 

The article is organized as follows. 
Notations and definitions are collected in Section \ref{sec:prelim}. 
Section \ref{sec:anisotropicgaborWF} recalls the definition of  
the anisotropic $s$-Gabor wave front set, and a result on tensorization is proved
as well as a characterization of the anisotropic $s$-Gabor wave front set
in terms of characteristic sets of symbols. 
Then Section \ref{sec:propagation} is devoted to a proof of the main result on propagation of the 
anisotropic $s$-Gabor wave front set. 
Finally Section \ref{sec:schrodinger} treats an application to a class of evolution equations of Schr\"odinger type.

\section{Preliminaries}\label{sec:prelim}

The unit sphere in $\rr d$ is denoted $\sr {d-1} \subseteq \rr d$. 
A ball of radius $r > 0$ centered in $x \in \rr d$ is denoted $\rB_r (x)$, and $\rB_r(0) = \rB_r$.  
The transpose of a matrix $A \in \rr {d \times d}$ is denoted $A^T$ and the inverse transpose 
of $A \in \GL(d,\ro)$ is $A^{-T}$. 
We write $f (x) \lesssim g (x)$ provided there exists $C>0$ such that $f (x) \leqs C \, g(x)$ for all $x$ in the domain of $f$ and of $g$. 
If $f (x) \lesssim g (x) \lesssim f(x)$ then we write $f \asymp g$. 
We use the bracket $\eabs{x} = (1 + |x|^2)^{\frac12}$ for $x \in \rr d$. 
Peetre's inequality with optimal constant \cite[Lemma~2.1]{Rodino3} is
\begin{equation}\label{eq:Peetre}
\eabs{x+y}^s \leqs \left( \frac{2}{\sqrt{3}} \right)^{|s|} \eabs{x}^s\eabs{y}^{|s|}\qquad x,y \in \rr d, \quad s \in \ro. 
\end{equation}
The normalization of the Fourier transform is
\begin{equation*}
 \cF f (\xi )= \widehat f(\xi ) = (2\pi )^{-\frac d2} \int _{\rr
{d}} f(x)e^{-i\scal  x\xi } \, \dd x, \qquad \xi \in \rr d, 
\end{equation*}
for $f\in \cS(\rr d)$ (the Schwartz space), where $\scal \cdo \cdo$ denotes the scalar product on $\rr d$. 
The conjugate linear action of a distribution $u$ on a test function $\phi$ is written $(u,\phi)$, consistent with the $L^2$ inner product $(\cdo ,\cdo ) = (\cdo ,\cdo )_{L^2}$ which is conjugate linear in the second argument. 

Denote translation by $T_x f(y) = f( y-x )$ and modulation by $M_\xi f(y) = e^{i \scal y \xi} f(y)$ 
for $x,y,\xi \in \rr d$ where $f$ is a function or distribution defined on $\rr d$. 
The composed operator is denoted $\Pi(x,\xi) = M_\xi T_x$. 
Let $\fy \in \cS(\rr d) \setminus \{0\}$. 
The short-time Fourier transform (STFT) of a tempered distribution $u \in \cS'(\rr d)$ is defined by 
\begin{equation*}
V_\fy u (x,\xi) = (2\pi )^{-\frac d2} (u, M_\xi T_x \fy) = \cF (u T_x \overline \fy)(\xi), \quad x,\xi \in \rr d. 
\end{equation*}
Then $V_\fy u$ is smooth and polynomially bounded \cite[Theorem~11.2.3]{Grochenig1}, that is 
there exists $k \geqs 0$ such that 
\begin{equation}\label{eq:STFTtempered}
|V_\fy u (x,\xi)| \lesssim \eabs{(x,\xi)}^{k}, \quad (x,\xi) \in T^* \rr d.  
\end{equation}
We have $u \in \cS(\rr d)$ if and only if
\begin{equation}\label{eq:STFTschwartz}
|V_\fy u (x,\xi)| \lesssim \eabs{(x,\xi)}^{-N}, \quad (x,\xi) \in T^* \rr d, \quad \forall N \geqs 0.  
\end{equation}

The inverse transform is given by
\begin{equation}\label{eq:STFTinverse}
u = (2\pi )^{-\frac d2} \iint_{\rr {2d}} V_\fy u (x,\xi) M_\xi T_x \fy \, \dd x \, \dd \xi
\end{equation}
provided $\| \fy \|_{L^2} = 1$, with action under the integral understood, that is 
\begin{equation}\label{eq:moyal}
(u, f) = (V_\fy u, V_\fy f)_{L^2(\rr {2d})}
\end{equation}
for $u \in \cS'(\rr d)$ and $f \in \cS(\rr d)$, cf. \cite[Theorem~11.2.5]{Grochenig1}. 

By \cite[Corollary~11.2.6]{Grochenig1} 
the topology for $\cS (\rr d)$ can be defined by the collection of seminorms
\begin{equation}\label{eq:seminormsS}
\cS(\rr d) \ni \psi \mapsto \| \psi \|_n := \sup_{z \in \rr {2d}} \eabs{z}^n |V_\fy \psi (z)|, \quad n \in \no,
\end{equation}
for any $\fy \in \cS(\rr d) \setminus 0$.

\subsection{$s$-conic subsets}\label{subsec:sconic}

We will use subsets of $T^* \rr d \setminus 0$ that are $s$-conic, 
that is closed under the operation $T^* \rr d \setminus 0 \ni (x,\xi) \mapsto ( \lambda x, \lambda^s \xi)$
for all $\lambda > 0$. 

Let $s > 0$ be fixed. 
We need the following simplified version of a tool taken from \cite{Parenti1} and its references. 
Given $(x,\xi) \in \rr {2d} \setminus 0$ there is a unique $\lambda = \lambda(x,\xi) = \lambda_s (x,\xi) > 0$ such that 
\begin{equation*}
\lambda (x,\xi)^{-2} | x |^2 + \lambda (x,\xi)^{-2s} | \xi |^2 = 1. 
\end{equation*}
Then $(x,\xi) \in \sr {2d-1}$ if and only if $\lambda (x,\xi) = 1$. 
By the implicit function theorem the function $\lambda: \rr {2d} \setminus 0 \to \ro_+$ is smooth \cite{Krantz1}. 
We have \cite[Eq.~(3.1)]{Rodino4}
\begin{equation}\label{eq:homogeneous1}
\lambda_s ( \mu x, \mu^s \xi) = \mu \lambda_s (x,\xi), \quad (x,\xi) \in \rr {2d} \setminus 0, \quad \mu > 0. 
\end{equation}

The projection $\pi_s(x,\xi)$ of $(x,\xi) \in \rr {2d} \setminus 0$ along the curve $\ro_+ \ni \mu \mapsto (\mu x, \mu^s \xi)$ onto $\sr {2d-1}$ is defined as
\begin{equation}\label{eq:projection}
\pi_s(x,\xi) = \left( \lambda(x,\xi)^{-1} x, \lambda(x,\xi)^{-s} \xi \right), \quad (x,\xi) \in \rr {2d} \setminus 0. 
\end{equation}
Then $\pi_s(\mu x, \mu^s \xi) = \pi_s(x, \xi)$ does not depend on $\mu > 0$. 
The function $\pi_s: \rr {2d} \setminus 0 \to \sr {2d-1}$ is smooth since $\lambda \in C^\infty(\rr {2d} \setminus 0)$
and $\lambda(x,\xi) > 0$ for all $(x,\xi) \in \rr {2d} \setminus 0$. 

From \cite{Parenti1}, or by straightforward arguments, we have the bounds
\begin{equation}\label{eq:lambdaboundnonisotropic}
|x| + |\xi|^{\frac1s}
\lesssim \lambda(x,\xi) \lesssim |x| + |\xi|^{\frac1s}, \quad (x,\xi) \in \rr {2d} \setminus 0
\end{equation}
and
\begin{equation}\label{eq:lambdaboundisotropic}
\eabs{ (x,\xi) }^{\min \left( 1, \frac1s \right)}
\lesssim 1 + \lambda(x,\xi) 
\lesssim \eabs{(x,\xi)}^{\max \left( 1, \frac1s \right)}, \quad (x,\xi) \in \rr {2d} \setminus 0. 
\end{equation}

We will use two types of $s$-conic neighborhoods. 
The first type is defined as follows. 

\begin{defn}\label{def:scone1}
Suppose $s, \ep > 0$ and $z_0 \in \sr {2d-1}$. 
Then
\begin{equation*}
\Gamma_{s, z_0, \ep}
= \{ (x,\xi) \in \rr {2d} \setminus 0, \ | z_0 - \pi_s(x,\xi) | < \ep \}
\subseteq  T^* \rr d \setminus 0. 
\end{equation*}
\end{defn}

We write $\Gamma_{z_0, \ep} = \Gamma_{s, z_0, \ep}$ when $s$ is fixed and understood from the context. 
If $\ep > 2$ then $\Gamma_{z_0, \ep} = T^* \rr d \setminus 0$ so we usually restrict to $\ep \leqs 2$. 

The second type of $s$-conic neighborhood is defined as follows. 

\begin{defn}\label{def:scone2}
Suppose $s, \ep > 0$ and $(x_0, \xi_0) \in \sr {2d-1}$. 
Then
\begin{align*}
\wt \Gamma_{(x_0,\xi_0),\ep}
& = \wt \Gamma_{s,(x_0,\xi_0),\ep} \\
& = \{ (y,\eta) \in \rr {2d} \setminus 0: \ (y,\eta) = (\lambda (x_0 + x), \lambda^s (\xi_0 + \xi), \ \lambda > 0, \ (x,\xi) \in \rB_\ep \} \\
& = \{ (y,\eta) \in \rr {2d} \setminus 0: \ \exists \lambda > 0: \ (\lambda y, \lambda^s \eta) \in (x_0,\xi_0) + \rB_\ep \}.  
\end{align*}
\end{defn}

By \cite[Lemma~3.7]{Rodino4} the two types of $s$-conic neighborhoods are topologically equivalent. 
This means that if $z_0 \in \sr {2d-1}$ then
for each $\ep > 0$ there exists $\delta > 0$ such that 
$\Gamma_{z_0, \delta} \subseteq \widetilde \Gamma_{z_0,\ep}$
and
$\wt \Gamma_{z_0, \delta} \subseteq \Gamma_{z_0,\ep}$.

\subsection{Pseudodifferential operators and anisotropic Shubin symbols}

We need some elements from the calculus of pseudodifferential operators \cite{Folland1,Hormander0,Nicola1,Shubin1}. 
Let $a \in C^\infty (\rr {2d})$, $m \in \ro$ and $0 \leqs \rho \leqs 1$. Then $a$ is a \emph{Shubin symbol} of order $m$ and parameter $\rho$, denoted $a\in G_\rho^m$, if for all $\alpha,\beta \in \nn d$ there exists a constant $C_{\alpha,\beta}>0$ such that
\begin{equation}\label{eq:shubinineq}
|\partial_x^\alpha \partial_\xi^\beta a(x,\xi)| \leqs C_{\alpha,\beta} \langle (x,\xi)\rangle^{m - \rho|\alpha + \beta|}, \quad x,\xi \in \rr d.
\end{equation}
The Shubin symbols $G_\rho^m$ form a Fr\'echet space where the seminorms are given by the smallest possible constants in \eqref{eq:shubinineq}.
We write $G_1^m = G^m$. 

For $a \in G_\rho^m$ and $t \in \ro$ a pseudodifferential operator in the $t$-quantization is defined by
\begin{equation}\label{eq:tquantization}
a_t(x,D) f(x)
= (2\pi)^{-d}  \int_{\rr {2d}} e^{i \langle x-y, \xi \rangle} a ( (1-t) x + t y,\xi ) \, f(y) \, \dd y \, \dd \xi, \quad f \in \cS(\rr d),
\end{equation}
when $m<-d$. The definition extends to $m \in \ro$ if the integral is viewed as an oscillatory integral.
If $t=0$ we get the Kohn--Nirenberg quantization $a_0(x,D)$ and if $t = \frac12$ we get the Weyl quantization $a_{1/2}(x,D) = a^w(x,D)$. 
The Weyl product is the product of symbols corresponding to operator composition (when well defined): 
$( a \wpr b)^w(x,D) = a^w(x,D) b^w (x,D)$.

Anisotropic versions of the Shubin classes are defined as follows 
\cite[Definition~3.1]{Rodino4}. 

\begin{defn}\label{def:symbol}
Let $s > 0$ and $m \in \ro$. 
The space of ($s$-)anisotropic Shubin symbols $G^{m,s}$ of order $m$ consists of functions $a \in C^\infty(\rr {2d})$ 
that satisfy the estimates
\begin{equation*}
|\pdd x \alpha \pdd \xi \beta a(x,\xi)|
\lesssim ( 1 + |x| + |\xi|^{\frac1s} )^{m - |\alpha| - s |\beta|}, \quad (x,\xi) \in T^* \rr d, \quad \alpha, \beta \in \nn d. 
\end{equation*}
\end{defn}

We have
\begin{equation*}
\bigcap_{m \in \ro} G^{m,s} = \cS(\rr {2d}),  
\end{equation*}
and $G^{m,1} = G^m = G_1^m$, that is the usual Shubin class, 
but we cannot embed $G_\rho^m$ in a space $G^{n,s}$ unless $\rho = s = 1$. 
Using \eqref{eq:lambdaboundnonisotropic} and \eqref{eq:lambdaboundisotropic} the embedding 
\begin{equation}\label{eq:Gmsinclusion}
G^{m,s} \subseteq G_\rho^{m_0}, 
\end{equation}
where $m_0 = \max(m, m/s)$ and $\rho = \min(s, 1/s)$, can be confirmed. 
Thus the Shubin calculus \cite{Shubin1,Nicola1} applies to the anisotropic Shubin symbols. 
But there is a more subtle anisotropic subcalculus adapted to the anisotropic Shubin symbols $G^{m,s}$, for each fixed $s > 0$.
In fact by \cite[Proposition~3.3]{Rodino4}
the symbol classes $G^{m,s}$ are invariant under a change of the quantization parameter $t \in \ro$ in \eqref{eq:tquantization}, 
and the Weyl product $\wpr: G^{m,s} \times G^{n,s} \to G^{m+n,s}$ is continuous.

The following two definitions are taken from \cite[Definitions~3.8 and 6.1]{Rodino4}. 
The anisotropic weight is denoted 
\begin{equation*}
\mu_s(x,\xi) = 1 + |x| + |\xi|^{\frac1s}. 
\end{equation*}

\begin{defn}\label{def:noncharacteristic}
Let $s > 0$, $z_0 \in \rr {2d} \setminus 0$, and $a \in G^{m,s}$. 
Then $z_0$ is called non-characteristic of order $m_1 \leqs m$, $z_0 \notin \charac_{s,m_1} (a)$, if there exists $\ep > 0$ such that, 
with $\Gamma = \Gamma_{s,\pi_s(z_0),\ep}$,
\begin{align}
|a( x, \xi )| & \geqs C \mu_s(x,\xi)^{m_1}, \quad (x,\xi) \in \Gamma \quad, \quad |x| + |\xi|^{\frac1s} \geqs R, \label{eq:lowerbound1} \\
|\pdd x \alpha \pdd \xi \beta a(x,\xi)| &\lesssim |a(x,\xi)| \mu_s(x,\xi)^{- |\alpha| - s |\beta|}, \quad \alpha, \beta \in \nn d, \quad (x,\xi) \in \Gamma, \quad |x| + |\xi|^{\frac1s} \geqs R, \label{eq:boundderivative1}
\end{align}
for suitable $C, R > 0$. 
\end{defn}

If $m_1 = m$ we write $\charac_{s,m} (a) = \charac_{s} (a)$, and then the condition \eqref{eq:boundderivative1} is redundant. 
Note that $\charac_{s,m_1} (a)$ is a closed $s$-conic subset of $T^* \rr d \setminus 0$, 
and $\charac_{s,m_1} (a) \subseteq \charac_{s,m_2} (a)$ if $m_1 \leqs m_2 \leqs m$. 

\begin{defn}\label{def:csupp}
Suppose $s > 0$, $a \in G^{m,s}$ and let $\pi_s$ be the projection \eqref{eq:projection}. 
The $s$-conical support $\csupp_{s} (a) \subseteq T^* \rr d \setminus 0$ of $a$ is
defined as follows. 
A point $z_0 \in T^* \rr d \setminus 0$ satisfies $z_0 \notin \csupp_{s} (a)$
if there exists $\ep > 0$ such that 
\begin{align*}
& \, \supp (a) \cap \overline{ \{ z \in \rr {2d} \setminus 0, \ | \pi_s(z) - \pi_s(z_0) | < \ep \} } \\
= & \, \supp (a) \cap \overline{ \Gamma }_{\pi_s (z_0), \ep} 
\quad \mbox{is compact in} \quad \rr {2d}. 
\end{align*}
\end{defn}

Clearly $\csupp_{s} (a) \subseteq T^* \rr d \setminus 0$ is $s$-conic. 
Note that for any $a \in G^{m,s}$ and any $m_1 \leqs m$ we have 
\begin{equation*}
\csupp_{s} (a) \cup \charac_{s,m_1} (a) = T^* \rr d \setminus 0. 
\end{equation*}
%

\section{Anisotropic Gabor wave front sets}\label{sec:anisotropicgaborWF}

The following definition is inspired by H.~Zhu's \cite[Definition~1.3]{Zhu1} of a quasi-homogen-eous 
wave front set defined by two non-negative parameters. 
Zhu uses a semiclassical formulation whereas we use the STFT. 
As far as we know it is an open question to determine if the concepts coincide.

Given a parameter $s > 0$ 
we define the $s$-Gabor wave front set $\WF_{\rm g}^{s} ( u ) \subseteq T^* \rr d \setminus 0$ of $u \in \cS'(\rr d)$. 

\begin{defn}\label{def:WFgs} 
Suppose $u \in \cS'(\rr d)$, $\fy \in \cS(\rr d) \setminus 0$, and $s > 0$. 
A point $z_0 = (x_0,\xi_0) \in T^* \rr d \setminus 0$ satisfies $z_0 \notin \WFgs ( u )$
if there exists an open set $U \subseteq T^* \rr d$ such that $z_0 \in U$ and 
\begin{equation}\label{eq:WFgs1}
\sup_{(x,\xi) \in U, \ \lambda > 0} \lambda^N |V_\fy u (\lambda x, \lambda^s \xi)| < + \infty \quad \forall N \geqs 0. 
\end{equation}
\end{defn}

If $s = 1$ we have $\WF_{\rm g}^{1} ( u ) = \WFg (u)$ 
which denotes the usual Gabor wave front set \cite{Hormander1,Rodino2}. 
We call $\WFgs ( u )$ the $s$-Gabor wave front set or the anisotropic Gabor wave front set. 
It is clear that $\WFgs ( u )$ is $s$-conic. 
In Definition \ref{def:WFgs} we may therefore assume that $(x_0,\xi_0) \in \sr {2d-1}$.

Referring to \eqref{eq:STFTtempered} and \eqref{eq:STFTschwartz} we see
that $\WF_{\rm g}^s ( u )$ records curves 
$0 < \lambda \mapsto (\lambda x, \lambda^s \xi)$ where $V_\fy u$ does not behave like the STFT of a Schwartz function. 
We have $\WFgs ( u ) = \emptyset$ if and only if $u \in \cS (\rr d)$ \cite[Section~4]{Rodino4}. 

If $s > 0$ then \eqref{eq:lambdaboundnonisotropic} and \eqref{eq:lambdaboundisotropic} give the bounds
\begin{equation}\label{eq:lambdaboundisotropic2}
\eabs{ (x,\xi) }^{\min \left( 1, \frac1s \right)}
\lesssim 1 + |x| + |\xi|^{\frac1s}
\lesssim \eabs{(x,\xi)}^{\max \left( 1, \frac1s \right)}, \quad (x,\xi) \in \rr {2d} \setminus 0. 
\end{equation}

If $(y,\eta) \in \wt \Gamma_{(x_0,\xi_0), \ep}$ for $0 < \ep < 1$ then for some $\lambda > 0$
and $(x,\xi) \in \rB_\ep$ we have
$(y,\eta) = (\lambda (x_0+x), \lambda^s (\xi_0+\xi))$. 
Thus $|y| + |\eta|^{\frac1s} \asymp \lambda$, so combining with \eqref{eq:lambdaboundisotropic2}
we obtain the following equivalent criterion to the condition \eqref{eq:WFgs1} in Definition \ref{def:WFgs}. 
The point $(x_0,\xi_0) \in \sr {2d-1}$ satisfies $(x_0,\xi_0) \notin \WFgs (u)$ if and only if for some $\ep > 0$
we have
\begin{equation}\label{eq:WFgs2}
\sup_{(x,\xi) \in \wt \Gamma_{(x_0,\xi_0), \ep}} \eabs{(x,\xi)}^N |V_\fy u (x,\xi)| < + \infty \quad \forall N \geqs 0. 
\end{equation}

We will need the following result on the anisotropic Gabor wave front set of a tensor product. 
The corresponding result for the Gabor wave front set is \cite[Proposition~2.8]{Hormander1}.
Here we use the notation $x=(x',x'') \in \rr {m+n}$, $x' \in \rr m$, $x'' \in \rr n$. 

\begin{prop}\label{prop:tensorWFs}
If $s > 0$, $u \in \cS'(\rr m)$, and $v \in \cS'(\rr n)$ then 
\begin{align*}
& \WFgs (u \otimes v) \subseteq \left( ( \WFgs(u) \cup \{0\} ) \times ( \WFgs(v) \cup \{0\} ) \right)\setminus 0 \\
& = \{ (x,\xi) \in T^* \rr {m+n} \setminus 0: \ (x',\xi') \in \WFgs(u) \cup \{ 0 \}, \ (x'',\xi'') \in \WFgs(v) \cup \{ 0 \} \} \setminus 0. 
\end{align*}
\end{prop}

\begin{proof}
Let $\fy \in \cS(\rr m) \setminus 0$ and $\psi \in \cS(\rr n) \setminus 0$. 
Suppose $(x_0,\xi_0) \in T^* \rr {m+n} \setminus 0$ does not belong to the set on the right hand side. 
Then either $(x_0',\xi_0') \notin \WFgs(u) \cup \{ 0 \}$ or $(x_0'',\xi_0'') \notin \WFgs(v) \cup \{ 0 \}$. 
For reasons of symmetry we may assume $(x_0',\xi_0') \notin \WFgs(u) \cup \{ 0 \}$. 

Thus there exists $\ep > 0$ such that 
\begin{equation*}
\sup_{(x',\xi') \in (x_0',\xi_0') + \rB_\ep, \ \lambda > 0} \lambda^N |V_\varphi u( \lambda x', \lambda^s \xi')| < \infty \quad \forall N \geqs 0. 
\end{equation*}

Let $(x',\xi') \in (x_0',\xi_0') + \rB_\ep$, $(x'',\xi'') \in (x_0'',\xi_0'') + \rB_\ep$, 
let $N \in \no$ be arbitrary, and let $\lambda \geqs 1$. 
We obtain using \eqref{eq:STFTtempered}, for some $k \in \no$
\begin{align*}
\lambda^N |V_{\fy \otimes \psi} u \otimes v (\lambda x, \lambda^s \xi)|
& = \lambda^N |V_\fy u ( \lambda x', \lambda^s \xi')| \, |V_\psi v (\lambda x'', \lambda^s \xi'')| \\
& \lesssim \lambda^N  |V_\fy u ( \lambda x', \lambda^s \xi')| \, \eabs{ (\lambda x'', \lambda^s \xi'' )}^k \\
& \leqs \lambda^{N+ k \max(1,s) } 
\left( 1 + ( |(x_0'', \xi_0'')| + \ep)^2  | \right)^{\frac{k}{2}} |V_\fy u ( \lambda x', \lambda^s \xi')|  \\
& \lesssim \lambda^{N+ k \max(1,s) } |V_\fy u ( \lambda x', \lambda^s \xi')|  < \infty. 
\end{align*}
It follows that $(x_0,\xi_0) \notin \WFgs (u \otimes v)$. 
\end{proof}

For the next result we need the following lemma to construct functions in $a \in G^{m,s}$
such that $\charac_{s,m} (a) = \emptyset$. 

\begin{lem}\label{lem:Gmsnonchar}
If $s > 0$ and $m \in \ro$ then there exists $a \in G^{m,s}$
such that $\charac_{s,m} (a) = \emptyset$. 
\end{lem}

\begin{proof}
Let $g \in C^\infty(\ro)$ satisfy $0 \leqs g \leqs 1$, $g(x) = 0$ if $x \leqs \frac12$ and $g(x) = 1$ if $x \geqs 1$.  
Set 
\begin{equation}\label{eq:homogeneous2}
\psi (\lambda x, \lambda^s \xi) = \lambda^m, \quad (x,\xi) \in \sr {2d-1}, \quad \lambda > 0, 
\end{equation}
and 
\begin{equation}\label{eq:cutoff1}
a (z) = g ( |z| ) \psi (z), \quad z \in \rr {2d}. 
\end{equation}

Note that \eqref{eq:homogeneous2} can be written
\begin{equation*}
\psi (x,\xi) = \lambda_s^m(x,\xi), \quad (x,\xi) \in \rr {2d} \setminus 0, 
\end{equation*}
and it follows that $\psi \in C^\infty( \rr {2d} \setminus 0 )$, 
and thus $a \in C^\infty (\rr {2d})$. 

If $(x,\xi) \in \rr {2d} \setminus 0$ and $\lambda > 0$ then by 
\eqref{eq:homogeneous1}
\begin{equation*}
\psi (\lambda x, \lambda^s \xi) = \lambda_s^m( \lambda x, \lambda^s \xi )
= \lambda^m  \psi( x, \xi ). 
\end{equation*}
This gives 
\begin{equation}\label{eq:homogeneous3}
(\pdd x \alpha \pdd \xi \beta \psi) (\lambda x, \lambda^s \xi)
= \lambda^{m-|\alpha| - s |\beta|}  \pdd x \alpha \pdd \xi \beta \psi( x, \xi ), \quad (x,\xi) \in \rr {2d} \setminus 0, 
\quad \lambda > 0, \quad \alpha, \beta \in \nn d. 
\end{equation}

Let $(y,\eta) \in \rr {2d} \setminus \rB_1$.
Then $(y,\eta) = (\lambda x, \lambda^s \xi)$ for a unique $(x,\xi) \in \sr {2d-1}$ and 
$\lambda = \lambda_s (y,\eta) \geqs 1$.
Combining
\begin{equation*}
1 + |y| + |\eta|^{\frac1s} = 1 + \lambda ( |x| + |\xi|^{\frac1s} ) \asymp 1+ \lambda
\end{equation*}
with \eqref{eq:homogeneous3} we obtain for any $\alpha, \beta \in \nn d$
\begin{equation*}
\left| \pdd y \alpha \pdd \eta \beta \psi (y, \eta) \right|
\leqs C_{\alpha,\beta} (1+\lambda)^{m -|\alpha| - s |\beta|} 
\lesssim ( 1 + |y| + |\eta|^{\frac1s} )^{m -|\alpha| - s |\beta|}. 
\end{equation*}
Referring to \eqref{eq:cutoff1} we may conclude that $a \in G^{m,s}$. 

For the same reason we have 
\begin{equation*}
\left| a (y, \eta) \right|
= \lambda^m 
\asymp ( 1 + |y| + |\eta|^{\frac1s} )^m, \quad |(y,\eta)| \geqs 1,  
\end{equation*}
which shows that $\charac_{s,m} (a) = \emptyset$. 
\end{proof}

\begin{rem}\label{rem:correction}
The proof of Lemma \ref{lem:Gmsnonchar} gives a correction of the 
slightly erroneous argument in the proof of \cite[Lemma~3.5]{Rodino4}. 
More precisely  \cite[Eq.~(3.16)]{Rodino4} is not well motivated. 
But the conclusion $\chi \in G^{0,s}$ follows from a homogeneity argument as above. 
\end{rem}

The following result generalizes \cite[Definitions~2.6 and 3.1 combined with Theorems~4.1 and 4.2]{Rodino2}, 
and is a characterization of the $s$-Gabor wave front set which is conceptually similar to characterizations of 
other types of wave front sets \cite{Hormander0}. 

\begin{prop}\label{prop:WFcharacterization}
If $s > 0$, $m \in \ro$ 
and $u \in \cS'(\rr d)$ then 
\begin{equation*}
\WFgs (u) = \bigcap_{a \in G^{m,s}: \ a^w(x,D) u \in \cS} \charac_{s,m} (a). 
\end{equation*}
\end{prop}

\begin{proof}
First we show 
\begin{equation}\label{eq:WFcharac2}
\WFgs (u) \subseteq \bigcap_{a \in G^{m,s}: \ a^w(x,D) u \in \cS} \charac_{s,m} (a). 
\end{equation}

Suppose 
$a \in G^{m,s}$, $a^w(x,D) u \in \cS$, $z_0 \in T^* \rr d \setminus 0$ and $z_0 \notin \charac_{s,m} (a)$. 
We may assume $| z_0 | = 1$. 
Let $\ep > 0$ be small enough to guarantee $\Gamma_{z_0, 2\ep} \cap \charac_{s,m} (a) = \emptyset$. 
By \cite[Lemma~3.5]{Rodino4} there exists for any $\rho > 0$ an $s$-conic cutoff function $\chi \in G^{0,s}$ such that 
$0 \leqs \chi \leqs 1$,
$\supp \chi \subseteq \Gamma_{z_0, 2\ep} \setminus \rB_{\rho/2}$
and $\chi |_{\Gamma_{z_0, \ep} \setminus \overline \rB_{\rho} } \equiv 1$. 

If $\rho > 0$ is sufficiently large then by \cite[Lemma~6.3]{Rodino4} there exists $b \in G^{-m,s}$ and $r \in \cS(\rr {2d})$ such that 
\begin{equation*}
b \wpr a = \chi - r. 
\end{equation*}
Thus we may write 
\begin{equation*}
u = (1-\chi)^w(x,D) u + b^w(x,D) a^w(x,D) u + r^w(x,D) u 
\end{equation*}
where $r^w(x,D) u \in \cS$ since $r^w(x,D): \cS' \to \cS$ is regularizing,  
and $b^w(x,D) a^w(x,D) u \in \cS$ since $a^w(x,D) u \in \cS$  
and $b^w(x,D): \cS \to \cS$ is continuous \cite[Section~23.2]{Shubin1}. 
It follows that $\WFgs (u) = \WFgs ( (1-\chi)^w(x,D) u )$,
and finally \cite[Proposition~6.2]{Rodino4} yields
\begin{equation*}
\WFgs (u) = \WFgs ( (1-\chi)^w(x,D) u ) \subseteq \csupp_{s} ( 1-\chi ) \subseteq T^* \rr d \setminus \Gamma_{z_0, \ep}. 
\end{equation*}
It follows that $z_0 \notin \WFgs (u)$ so we have proved \eqref{eq:WFcharac2}. 

It remains to show 
\begin{equation}\label{eq:WFcharac3}
\WFgs (u) \supseteq \bigcap_{a \in G^{m,s}: \ a^w(x,D) u \in \cS} \charac_{s,m} (a). 
\end{equation}

Suppose $z_0 \in T^* \rr d \setminus 0$, $z_0 \notin \WFgs (u)$ and $| z_0 | = 1$. 
Let $\ep > 0$ be small enough to guarantee $\Gamma_{z_0, 2\ep} \cap \WFgs (u) = \emptyset$. 
Let $\rho > 0$ and let $\chi \in G^{0,s}$ satisfy
$0 \leqs \chi \leqs 1$,
$\supp \chi \subseteq \Gamma_{z_0, 2\ep} \setminus \rB_{\rho/2}$
and $\chi |_{\Gamma_{z_0, \ep} \setminus \overline \rB_{\rho} } \equiv 1$. 
Using Lemma \ref{lem:Gmsnonchar} we let $b \in G^{m,s}$ satisfy $\charac_{s,m} (b) = \emptyset$, 
and we set $a = b \chi \in G^{m,s}$. 
Then $z_0 \notin \charac_{s,m} ( a )$. 

We have $\csupp_{s} ( a ) \subseteq \Gamma_{z_0, 2\ep}$, and by the microlocal inclusion \cite[Proposition~5.1]{Rodino4}
we have $\WFgs ( a^w (x,D) u) \subseteq \WFgs (u)$. 
Combining with \cite[Proposition~6.2]{Rodino4} this implies 
\begin{equation*}
\WFgs ( a^w (x,D) u ) \subseteq \csupp_{s} ( a ) \cap \WFgs (u) 
\subseteq \Gamma_{z_0, 2\ep} \cap \WFgs (u) =  \emptyset. 
\end{equation*}
It follows that $a^w (x,D) u \in \cS$, which means that we have proved \eqref{eq:WFcharac3}. 
\end{proof}

\section{Propagation of anisotropic Gabor wave front sets}\label{sec:propagation}

Define for $K \in \cS'(\rr {2d})$
\begin{align*}
\WF_{\rm g,1}^s(K) & = \{ (x,\xi) \in T^* \rr d: \ (x, 0, \xi, 0) \in \WFgs (K) \} & \subseteq T^* \rr d \setminus 0, \\
\WF_{\rm g,2}^s(K) & = \{ (y,\eta) \in T^* \rr d: \ (0, y, 0, -\eta) \in \WFgs (K) \} & \subseteq T^* \rr d \setminus 0. 
\end{align*}

We will use the assumption
\begin{equation}\label{eq:WFKjempty}
\WF_{\rm g,1}^s (K) = \WF_{\rm g,2}^s (K) = \emptyset. 
\end{equation}  

We note that the condition \eqref{eq:WFKjempty} appears in several other works for 
various global isotropic \cite{Carypis1,Hormander1,PRW1,Wahlberg1} and anisotropic \cite{Wahlberg3} wave front sets. 
The following lemma is 
a version of \cite[Lemma~5.1]{Wahlberg3} for tempered distributions and the $s$-Gabor wave front set (cf. \cite[Lemma~6.1]{Carypis1}). 

\begin{lem}\label{lem:WFkernelaxes}
If $s > 0$, $K \in \cS'(\rr {2d})$ and \eqref{eq:WFKjempty} holds, then there exists $c > 1$ such that 
\begin{equation}\label{eq:Gamma1}
\WF_{\rm g}^s (K) \subseteq 
\Gamma_1 := \left\{ (x,y,\xi,\eta) \in T^* \rr {2d}: \ c^{-1} \left( |x| + |\xi|^{\frac1s} \right) <  |y| + |\eta|^{\frac1s}  < c \left( |x| + |\xi|^{\frac1s} \right)  \right\}. 
\end{equation}
\end{lem}

\begin{proof}
Suppose that
\begin{equation*}
\WFgs (K) \subseteq \left\{ (x,y,\xi,\eta) \in T^* \rr {2d}: \ |y| + |\eta|^{\frac1s}  < c \left( |x| + |\xi|^{\frac1s} \right) \right\}
\end{equation*}
does not hold for any $c > 0$. Then for each $n \in \no$ there exists $(x_n,y_n,\xi_n,\eta_n) \in \WFgs (K)$ such that 
\begin{equation}\label{eq:phasespaceineq1}
|y_n| + |\eta_n|^{\frac1s} \geqs n \left( |x_n| + |\xi_n|^{\frac1s} \right).
\end{equation}
By rescaling $(x_n,y_n,\xi_n,\eta_n)$ as
$(x_n,y_n,\xi_n,\eta_n) \mapsto ( \lambda x_n, \lambda y_n, \lambda^{s} \xi_n, \lambda^{s} \eta_n)$
we obtain for a unique $\lambda = \lambda (x_n,y_n,\xi_n,\eta_n) > 0$ a vector in $\WFgs (K) \cap \sr {4d-1}$, 
cf. Section \ref{subsec:sconic}. 
This $s$-conic rescaling leaves \eqref{eq:phasespaceineq1} invariant. 
Abusing notation we still denote the rescaled vector $(x_n,y_n,\xi_n,\eta_n) \in \WFgs (K) \cap \sr {4d-1}$.

From \eqref{eq:phasespaceineq1}  it follows that $(x_n,\xi_n) \rightarrow 0$ as $n \rightarrow \infty$. 
Passing to a subsequence (without change of notation) and using the closedness of $\WFgs (K)$ gives
\begin{equation*}
(x_n,y_n,\xi_n,\eta_n) \rightarrow (0,y,0,\eta) \in \WFgs (K), \quad n \rightarrow \infty, 
\end{equation*}
for some $(y,\eta) \in \sr {2d-1}$. This implies $(y,-\eta) \in \WF_{\rm{g}, 2}^s(K)$ which is a contradiction. 

Similarly one shows 
\begin{equation*}
\WFgs (K) \subseteq \left\{ (x,y,\xi,\eta) \in T^* \rr {2d}: \  |x| + |\xi|^{\frac1s}  <  c \left( |y| + |\eta|^{\frac1s} \right) \right\}
\end{equation*}
for some $c > 0$ using $\WF_{\rm g,1}^s(K) = \emptyset$. 
\end{proof}

The set $\Gamma_1 \subseteq \rr {4d} \setminus 0$ in \eqref{eq:Gamma1} is open, and $s$-conic in the sense that it is closed with respect
to $(x,y,\xi,\eta) \mapsto ( \lambda x, \lambda y, \lambda^s \xi, \lambda^s \eta )$ for any $\lambda > 0$. 
Hence $(\rr {4d} \setminus \Gamma_1)$ is $s$-conic and $(\rr {4d} \setminus \Gamma_1) \cap \sr{4d-1}$ is compact. 
From \eqref{eq:WFgs2} we then obtain if $\Phi \in \cS(\rr {2d}) \setminus 0$
\begin{equation}\label{eq:WFKcomplement}
| V_\Phi K( x, y, \xi, - \eta) |  
\lesssim \eabs{(x,y,\xi,\eta)}^{-m}, \quad m \in \no, \quad ( x, y, \xi, - \eta) \in \rr {4d} \setminus \Gamma_1. 
\end{equation}

From \eqref{eq:Gamma1} and \eqref{eq:lambdaboundisotropic} it follows that 
\begin{equation}\label{eq:Gamma1proportionality}
(x,y,\xi, - \eta) \in \Gamma_1 \quad \Longrightarrow \quad
\eabs{(y,\eta)}^{\min \left( s, \frac1s\right)}
\lesssim \eabs{(x,\xi)}
\lesssim \eabs{(y,\eta)}^{\max \left( s, \frac1s\right)}. 
\end{equation}

A tempered distribution 
$K \in \cS'(\rr {2d})$ defines a continuous linear map $\cK: \cS (\rr d) \to \cS'(\rr d)$ by
\begin{equation}\label{eq:kernelop}
(\cK f, g) = (K, g \otimes \overline f), \quad f,g \in \cS(\rr d). 
\end{equation}

The following result says that the condition \eqref{eq:WFKjempty} implies continuity of $\cK$ on $\cS(\rr d)$
and a unique extension to a continuous operator on $\cS'(\rr d)$. 
This is the basis for the forthcoming result on propagation of the $s$-Gabor wave front sets Theorem \ref{thm:WFgspropagation}. 
In the proof we use the conventional notation (cf. \cite{Hormander1,Hormander2}) for the reflection operator in the fourth $\rr d$ coordinate in $\rr {4d}$
\begin{equation}\label{eq:reflection}
(x,y,\xi,\eta)' = (x,y,\xi,-\eta), \quad x,y,\xi,\eta \in \rr d. 
\end{equation}

\begin{prop}\label{prop:opSTFTformula}
Let $s > 0$ and let $\cK: \cS(\rr d) \to \cS'(\rr d)$ be the continuous linear operator \eqref{eq:kernelop}
defined by the Schwartz kernel $K \in \cS'(\rr {2d})$. 
If \eqref{eq:WFKjempty} holds then 

\begin{enumerate}[\rm (i)]

\item $\cK: \cS(\rr d) \to \cS(\rr d)$ is continuous; 

\item $\cK$ extends uniquely to a sequentially continuous linear operator $\cK: \cS' (\rr d) \to \cS' (\rr d)$; 

\item if $\fy \in \cS(\rr d)$, $\| \fy \|_{L^2} = 1$, $\Phi = \fy \otimes \fy \in \cS(\rr {2d})$, $u \in \cS'(\rr d)$ and $\psi \in \cS(\rr d)$, then
\begin{equation}\label{eq:opSTFT1}
(\cK u, \psi) 
= \int_{\rr {4d}} V_\Phi K(x,y,\xi,-\eta) \, \overline{V_\fy \psi (x,\xi)} \, V_{\overline \fy} u(y,\eta) \, \dd x \, \dd y \, \dd \xi \, \dd \eta. 
\end{equation}

\end{enumerate}
\end{prop}

\begin{proof}
By \cite[Lemma~5.1]{Wahlberg1} the formula \eqref{eq:opSTFT1} holds for $u,\psi \in \cS(\rr d)$. 

Let $\fy \in \cS(\rr d)$ satisfy $\| \fy \|_{L^2} = 1$ and set $\Phi = \fy \otimes \fy \in \cS(\rr {2d})$. 
Since
\begin{equation*}
\overline{V_\fy \Pi(x,\xi) \fy (y,\eta)} = e^{i \la y, \eta - \xi \ra} V_\fy \fy ( x-y, \xi - \eta)
\end{equation*}
we get from \eqref{eq:opSTFT1} for $u \in \cS (\rr d)$ and $(x,\xi) \in T^* \rr d$
\begin{equation}\label{eq:STFTKu}
\begin{aligned}
& V_\fy(\cK u) (x, \xi) 
= (2 \pi)^{-\frac{d}{2}} (\cK u, \Pi(x,\xi) \fy) \\
& = (2 \pi)^{-\frac{d}{2}} \int_{\rr {4d}} e^{i \la y,\eta -\xi \ra} V_\Phi K (y,z,\eta,-\theta) V_\fy \fy (x-y,\xi-\eta) \, V_{\overline \fy} u(z,\theta) \, \dd y \, \dd z \, \dd \eta \, \dd \theta
\end{aligned}
\end{equation}
which gives 
\begin{equation}\label{eq:STFTKuabs}
|V_\fy(\cK u) (x, \xi)|
\lesssim \int_{\rr {4d}} | V_\Phi K (y,z,\eta,-\theta) | \, | V_\fy \fy (x-y,\xi-\eta)| \, | V_{\overline \fy} u(z,\theta)| \, \dd y \, \dd z \, \dd \eta \, \dd \theta. 
\end{equation}

We use the seminorms \eqref{eq:seminormsS} for $\cS(\rr d)$.
Let $n \in \no$ and
consider first the right hand side integral in \eqref{eq:STFTKuabs} over $(y,z,\eta,-\theta) \in \rr {4d} \setminus \Gamma_1$
where $\Gamma_1$ is defined by \eqref{eq:Gamma1} with $c > 1$ chosen 
so that $\WFgs (K) \subseteq \Gamma_1$.
By Lemma \ref{lem:WFkernelaxes} we may use the estimates \eqref{eq:WFKcomplement}. 
Using \eqref{eq:Peetre} and \eqref{eq:STFTschwartz} we obtain for any $m \in \no$
\begin{equation}\label{eq:STFTKuabs1}
\begin{aligned}
& \int_{\rr {4d} \setminus \Gamma_1'} 
|V_\Phi K(y,z,\eta,-\theta)| \,  | V_\fy \fy (x-y,\xi-\eta)|  \, |V_{\overline \fy} u (z,\theta)| \, \dd y \, \dd z \, \dd \eta \, \dd \theta \\
& \lesssim 
\int_{\rr {4d} \setminus \Gamma_1'} 
\eabs{(y,z,\eta,\theta)}^{-m} \,  \eabs{(x-y,\xi-\eta)}^{-n} \, |V_{\overline \fy} u (z,\theta)| \, \dd y \, \dd z \, \dd \eta \, \dd \theta \\
& \lesssim 
\| u \|_0 \eabs{(x,\xi)}^{-n} 
\int_{\rr {4d}}  \eabs{(y,z,\eta,\theta)}^{n-m}\, \dd y \, \dd z \, \dd \xi \, \dd \eta \\
& \lesssim \| u \|_0 \eabs{(x,\xi)}^{-n}
\end{aligned}
\end{equation}
provided $m > n + 4 d$. 

Next we consider the right hand side integral \eqref{eq:STFTKuabs} over $(y,z,\eta,-\theta) \in \Gamma_1$. 
Then we may 
use \eqref{eq:Gamma1proportionality}. 
From \eqref{eq:STFTtempered}  and \eqref{eq:STFTschwartz} we obtain for some $m \geqs 0$
and any $k \geqs 0$
\begin{equation}\label{eq:STFTKuabs2}
\begin{aligned}
& \int_{\Gamma_1'} 
|V_\Phi K(y,z,\eta,-\theta)| \,  | V_\fy \fy (x-y,\xi-\eta)|  \, |V_{\overline \fy} u (z,\theta)| \, \dd y \, \dd z \, \dd \eta \, \dd \theta \\
& \lesssim 
\| u \|_k
\eabs{(x,\xi)}^{-n} 
\int_{\Gamma_1'} 
\eabs{(y,z,\eta,\theta)}^{m+4d+1-4d-1} \, \eabs{(y,\eta)}^{n}  \,\eabs{(z,\theta)}^{-k}  \, \dd x \, \dd y \, \dd \xi \, \dd \eta \\
& \lesssim 
\| u \|_k \eabs{(x,\xi)}^{-n} 
\int_{\Gamma_1'} 
\eabs{(y,z,\eta,\theta)}^{-4d-1} \eabs{(z,\theta)}^{(m+4d+1) \left(1 + \max \left( s, \frac1s \right) \right) + n \max \left( s, \frac1s \right) - k } 
\dd x \dd y \dd \xi \dd \eta \\
& \lesssim 
\| u \|_k \eabs{(x,\xi)}^{-n} 
\end{aligned}
\end{equation}
provided $k > 0$ is sufficiently large. 

Combining \eqref{eq:STFTKuabs1} and \eqref{eq:STFTKuabs2} we obtain from \eqref{eq:STFTKuabs}
$\| \cK u \|_n \lesssim \| u \|_k$, which proves claim (i). 

To show claims (ii) and (iii) let $u \in \cS'(\rr d)$ and set 
for $N \in \no$
\begin{equation*}
u_N = (2 \pi)^{-\frac{d}{2}} \int_{|z| \leqs N} V_\fy u(z) \Pi(z) \fy \, \dd z. 
\end{equation*}

From \eqref{eq:STFTtempered} for some $k \geqs 0$, and \eqref{eq:STFTschwartz}
we obtain for any $n \geqs 0$
\begin{align*}
\eabs{w}^n |V_\fy u_N (w)| 
& \lesssim \int_{|z| \leqs N} |V_\fy u(z)| \, \eabs{w}^n |V_\fy \fy(w-z)| \, \dd z \\
& \lesssim \int_{|z| \leqs N} \eabs{z}^k \, \eabs{w}^n \, \eabs{w-z}^{-n} \, \dd z \\
& \lesssim \int_{|z| \leqs N} \eabs{z}^{k+n} \, \dd z 
\leqs C_{N,n}, \quad w \in \rr {2d}. 
\end{align*}
Referring to the seminorms \eqref{eq:seminormsS} shows 
that $u_N \in \cS(\rr d)$ for $N \in \no$. 
The fact that $u_N \to u$ in $\cS'(\rr d)$ as $N \to \infty$ is a consequence of 
\eqref{eq:moyal}, 
\eqref{eq:STFTtempered}, \eqref{eq:STFTschwartz} and dominated convergence. 

We also need the estimate (cf. \cite[Eq.~(11.29)]{Grochenig1})
\begin{equation*}
|V_{\overline{\fy}} u_N (z)| \leqs (2 \pi)^{-\frac{d}{2}} |V_\varphi u| * |V_{\overline{\fy}} \fy| (z), \quad z \in \rr {2d}, 
\end{equation*}
which in view of \eqref{eq:STFTtempered} and \eqref{eq:STFTschwartz}
gives the bound
\begin{equation}\label{eq:STFTupperbound3}
|V_{\overline{\fy}} u_N (z)| \lesssim \eabs{z}^{k + 2d + 1}, \quad z \in \rr {2d}, \quad N \in \no, 
\end{equation}
that holds uniformly over $N \in \no$, for some $k \in \no$. 

We are now in a position to assemble the ingredients into a proof of formula 
\eqref{eq:opSTFT1} for $u \in \cS'(\rr d)$ and $\psi \in \cS(\rr d)$. 
Set
\begin{equation}\label{eq:STFTequalitylimit}
\begin{aligned}
(\cK u, \psi) 
& = \lim_{N \to \infty} (\cK u_N, \psi) \\
& = \lim_{N \to \infty} 
\int_{\rr {4d}} V_\Phi K(x,y,\xi,-\eta) \, \overline{V_\fy \psi (x,\xi)} \, V_{\overline \fy} u_N (y,\eta) \, \dd x \, \dd y \, \dd \xi \, \dd \eta.
\end{aligned}
\end{equation}
Since $V_{\overline \fy} u_N(y,\eta) \to V_{\overline \fy} u(y,\eta)$ as $N \to \infty$ for all $(y,\eta) \in \rr {2d}$, the formula \eqref{eq:opSTFT1} follows from dominated convergence if we can show that the modulus of the integrand in \eqref{eq:STFTequalitylimit} is bounded by an integrable function that does not depend on $N \in \no$, which we now set out to do. 

Consider first the right hand side integral over $(x,y,\xi,-\eta) \in \rr {4d} \setminus \Gamma_1$
where $\Gamma_1$ is defined by \eqref{eq:Gamma1} with $c > 1$ again chosen 
so that $\WFgs (K) \subseteq \Gamma_1$.
By Lemma \ref{lem:WFkernelaxes} we may use the estimates \eqref{eq:WFKcomplement}. 
Using \eqref{eq:STFTupperbound3} we obtain for any $m \in \no$
\begin{equation}\label{eq:seminormest1}
\begin{aligned}
& \int_{\rr {4d} \setminus \Gamma_1'} 
|V_\Phi K(x,y,\xi,-\eta)| \,  |V_\fy \psi (x,\xi)| \, |V_{\overline \fy} u_N (y,\eta)| \, \dd x \, \dd y \, \dd \xi \, \dd \eta \\
& \lesssim 
\int_{\rr {4d} \setminus \Gamma_1'} 
\eabs{(x,y,\xi,\eta)}^{-m} \, |V_\fy \psi (x,\xi)| \, \eabs{(y,\eta)}^{k + 2 d + 1} \, \dd x \, \dd y \, \dd \xi \, \dd \eta \\
& \lesssim \sup_{z \in \rr {2d}} |V_\fy \psi (z)| 
\int_{\rr {4d}}  \eabs{(x,y,\xi,\eta)}^{k + 2 d + 1 - m} \, \dd x \, \dd y \, \dd \xi \, \dd \eta \\
& \lesssim \sup_{z \in \rr {2d}} |V_\fy \psi (z)| < \infty
\end{aligned}
\end{equation}
provided $m > 0$ is sufficiently large. 

Next we consider the right hand side integral \eqref{eq:STFTequalitylimit} over $(x,y,\xi,-\eta) \in \Gamma_1$
where we may 
use \eqref{eq:Gamma1proportionality}. 
Again from \eqref{eq:STFTtempered} we obtain for some $m \geqs 0$
\begin{equation}\label{eq:seminormest2}
\begin{aligned}
& \int_{\Gamma_1'} 
|V_\Phi K(x,y,\xi,-\eta)| \,  |V_\fy \psi (x,\xi)| \, |V_{\overline \fy} u_N (y,\eta)| \, \dd x \, \dd y \, \dd \xi \, \dd \eta \\
& \lesssim 
\int_{\Gamma_1'} 
\eabs{(x,y,\xi,\eta)}^{m+4d+1-4d-1} \, |V_\fy \psi (x,\xi)| \, \eabs{(y,\eta)}^{k + 2 d + 1} \, \dd x \, \dd y \, \dd \xi \, \dd \eta \\
& \lesssim 
\int_{\Gamma_1'} 
\eabs{(x,y,\xi,\eta)}^{-4d-1} \, \eabs{(x,\xi)}^{(m+6d+2+k) \left(1 + \max \left( s, \frac1s \right)\right)} \, |V_\fy \psi (x,\xi)| \, \dd x \, \dd y \, \dd \xi \, \dd \eta \\
& \lesssim \sup_{z \in \rr {2d}} \eabs{z}^{ (m+6d+2+k) \left(1 + \max \left( s, \frac1s \right)\right) } |V_\fy \psi (z)| < \infty. 
\end{aligned}
\end{equation}
The estimates \eqref{eq:seminormest1} and \eqref{eq:seminormest2} prove our claim that  
the modulus of the integrand in right hand side of \eqref{eq:STFTequalitylimit} is bounded by an $L^1(\rr {4d})$ function uniformly over $N \in \no$. 
Thus \eqref{eq:STFTequalitylimit} extends the domain of $\cK$ from $\cS(\rr d)$ to $\cS'(\rr d)$. 
We have shown claim (iii). 

From \eqref{eq:seminormest1} and \eqref{eq:seminormest2} we also see that 
$\cK$ extended to the domain $\cS'(\rr d)$ satisfies 
$\cK u \in \cS'(\rr d)$ when $u \in \cS'(\rr d)$. 
To prove claim (ii) it remains to show the sequential continuity of the extension  \eqref{eq:STFTequalitylimit} on $\cS'(\rr d)$. 
The uniqueness of the extension is a consequence of the continuity. 

Let $(u_n)_{n = 1}^\infty \subseteq \cS'(\rr d)$ be a sequence such that $u_n \to 0$ in $\cS'(\rr d)$ as $n \to \infty$. 
Then $V_{\overline \fy} u_n(y,\eta) \to 0$ as $n \to \infty$ for all $(y,\eta) \in \rr {2d}$. 
By the Banach--Steinhaus theorem \cite[Theorem~V.7]{Reed1}, $(u_n)_{n = 1}^\infty$ is equicontinuous. 
This means that there exists $m \in \no$ such that 
\begin{equation*}
|(u_n, \psi)| 
\lesssim \| \psi \|_m
= \sup_{w \in \rr {2d}} \eabs{w}^m |V_\fy \psi (w)|, \quad \psi \in \cS(\rr d), \quad n \in \no. 
\end{equation*}

Hence
\begin{align*}
|V_{\overline \fy} u_n(z)| 
& = (2 \pi)^{- \frac{d}{2}} |(u_n, \Pi(z) \overline \fy )| 
\lesssim \sup_{w \in \rr {2d}} \eabs{w}^m |V_\fy (\Pi(z) \overline \fy) (w)| \\
& = \sup_{w \in \rr {2d}} \eabs{w}^m |V_\fy \overline \fy (w-z) | 
\lesssim \sup_{w \in \rr {2d}} \eabs{w}^m \eabs{w-z}^{-m} 
\lesssim \eabs{z}^m, \quad z \in \rr {2d},
\end{align*}
uniformly for all $n \in \no$. 
From \eqref{eq:opSTFT1}, the estimates \eqref{eq:seminormest1}, \eqref{eq:seminormest2}, and dominated convergence it follows 
that $(\cK u_n, \psi) \to 0$ as $n \to \infty$ for all $\psi \in \cS(\rr d)$, that is $\cK u_n \to 0$ in $\cS'(\rr d)$. 
This finally proves claim (ii). 
\end{proof}

Now we start to prepare for the main result Theorem \ref{thm:WFgspropagation}. 
We need the relation mapping between a subset $A \subseteq X \times Y$ of the Cartesian product of two sets $X$, $Y$, and a subset $B \subseteq Y$, 
\begin{equation*}
A \circ B = \{ x \in X: \, \exists y \in B: \, (x,y) \in A \} \subseteq X.  
\end{equation*}
When $X = Y = \rr {2d}$ we use the convention
\begin{equation*}
A' \circ B  = \{ (x,\xi) \in \rr {2d}: \,  \exists (y,\eta) \in B: \, (x,y,\xi,-\eta) \in A \}. 
\end{equation*}
Note that there is a swap of the second and third variables. 

If we denote by 
\begin{align*}
\pi_{1,3}(x,y,\xi,\eta) & = (x,\xi), \\
\pi_{2,-4}(x,y,\xi,\eta) & = (y,-\eta), \quad x,y,\xi, \eta \in \rr d, 
\end{align*}
the projections $\rr {4d} \rightarrow \rr {2d}$ onto the first and the third $\rr d$ coordinate, 
and onto the second and the fourth $\rr d$ coordinate with a change of sign in the latter, respectively, then we may write
\begin{equation}\label{eq:relationproj}
\WFgs (K)' \circ \WFgs (u) 
= \pi_{1,3} \left( \WFgs (K) \cap \pi_{2,-4}^{-1} \WFgs (u) \right). 
\end{equation}

We need a lemma which is similar to \cite[Lemma~5.1]{Wahlberg3}. 

\begin{lem}\label{lem:sconicclosed}
If $s > 0$, $K \in \cS'(\rr {2d})$, \eqref{eq:WFKjempty} holds and $u \in \cS'(\rr d)$
then 
\begin{equation*}
\WF_{\rm g}^s (K)' \circ \WF_{\rm g}^s (u) \subseteq T^* \rr d \setminus 0
\end{equation*}
is $s$-conic and closed in $T^* \rr d \setminus 0$. 
\end{lem}

\begin{proof}
Let $(x,\xi) \in \WFgs (K)' \circ \WFgs (u)$. 
Then there exists $(y,\eta) \in \WFgs (u)$ such that $(x,y,\xi,-\eta) \in \WFgs (K)$. 
Let $\lambda > 0$.
Since $\WFgs (K)$ and $\WFgs (u)$ are $s$-conic we have 
$( \lambda x, \lambda y, \lambda^s \xi,- \lambda^s \eta) \in \WFgs (K)$
and $(\lambda y, \lambda^s \eta) \in \WFgs (u)$. 
It follows that $(\lambda x, \lambda^s \xi) \in \WFgs (K)' \circ \WFgs (u)$ 
which shows that $\WFgs (K)' \circ \WFgs (u)$ is $s$-conic. 

Next we assume that $(x_n,\xi_n) \in \WFgs (K)' \circ \WFgs (u)$ for $n \in \no$
and $(x_n, \xi_n) \to (x,\xi) \neq 0$ as $n \to +\infty$. 
For each $n \in \no$ there exists $(y_n,\eta_n) \in \WFgs (u)$ such that 
$(x_n,y_n,\xi_n,-\eta_n) \in \WFgs (K)$. 

Since the sequence $\{ (x_n, \xi_n)_n \} \subseteq T^* \rr d$ is bounded it follows from Lemma \ref{lem:WFkernelaxes}
that also the sequence $\{ (y_n, \eta_n)_n \} \subseteq T^* \rr d$ is bounded. 
Passing to a subsequence (without change of notation) we get convergence
\begin{equation*}
\lim_{n \to +\infty} (x_n,y_n,\xi_n,-\eta_n) 
= (x,y,\xi,-\eta) \in \rr {4d} \setminus 0. 
\end{equation*}
Here $(x,y,\xi,-\eta) \in \WFgs (K)$ since $\WFgs (K) \subseteq T^* \rr {2d} \setminus 0$ is closed, and 
$(y,\eta) \neq 0$ due to the assumption $\WF_{\rm g,1}^{s} (K) = \emptyset$. 
Moreover $(y,\eta) \in \WFgs (u)$ since $\WFgs (u) \subseteq T^* \rr d \setminus 0$ is closed, 
We have proved that $(x,\xi) \in \WFgs (K)' \circ \WFgs (u)$ which shows that $\WFgs (K)' \circ \WFgs (u)$ is closed in $T^* \rr d \setminus 0$. 
\end{proof}

Finally we may state and prove our main result on propagation of singularities.

\begin{thm}\label{thm:WFgspropagation}
Let $s > 0$ and let $\cK: \cS(\rr d) \to \cS'(\rr d)$ be the continuous linear operator \eqref{eq:kernelop}
defined by the Schwartz kernel $K \in \cS'(\rr {2d})$, and suppose that \eqref{eq:WFKjempty} holds. 
Then for $u \in \cS'(\rr d)$
we have
\begin{equation*}
\WF_{\rm g}^s (\cK u) \subseteq \WF_{\rm g}^s (K)' \circ \WF_{\rm g}^s (u).  
\end{equation*}
\end{thm}

\begin{proof}
By Proposition \ref{prop:opSTFTformula} $\cK: \cS (\rr d) \to \cS(\rr d)$ is continuous and extends
uniquely to a continuous linear operator $\cK: \cS' (\rr d) \to \cS'(\rr d)$. 

Let $\fy \in \cS(\rr d)$ satisfy $\| \fy \|_{L^2} = 1$ and set $\Phi = \fy \otimes \fy \in \cS(\rr {2d})$. 
Proposition \ref{prop:opSTFTformula}, \eqref{eq:opSTFT1} and \eqref{eq:STFTKu} give
for $u \in \cS' (\rr d)$ and $(x,\xi) \in T^* \rr d$ and $\lambda > 0$
\begin{equation}\label{eq:opSTFTest1}
| V_\fy(\cK u) ( \lambda x, \lambda^s \xi) | 
\lesssim \int_{\rr {4d}} | V_\Phi K (y,z,\eta,-\theta) | \, |V_\fy \fy ( \lambda x-y, \lambda^s \xi-\eta) | \, | V_{\overline \fy} u(z,\theta) | \, \dd y \, \dd z \, \dd \eta \, \dd \theta. 
\end{equation}

Suppose $z_0 = (x_0,\xi_0) \in T^* \rr d \setminus 0$ and
\begin{equation}\label{eq:notWFs1}
z_0 \notin \WFgs (K)' \circ \WFgs (u).
\end{equation}
To prove the theorem we will show $z_0\notin \WFgs (\cK u)$. 

By Lemma \ref{lem:sconicclosed} the set $\WFgs (K)' \circ \WFgs (u)$ is $s$-conic and closed. 
Thus we may assume that $z_0 \in \sr {2d-1}$.
Moreover, with $\wt \Gamma_{z_0,2 \ep} = \wt \Gamma_{s, z_0, 2 \ep}$,
there exists $\ep > 0$ such that 
\begin{equation*}
\overline{\wt \Gamma}_{z_0,2 \ep} \cap \left( \WFgs (K)' \circ \WFgs (u) \right)= \emptyset. 
\end{equation*}
Here $\overline{\wt \Gamma}_{z_0,2 \ep}$ denotes the closure of $\wt \Gamma_{z_0,2 \ep}$ in $T^* \rr d \setminus 0$. 
Using \eqref{eq:relationproj} we may write this as 
\begin{equation*}
\overline{\wt \Gamma}_{z_0,2 \ep} \cap \pi_{1,3} \left( \WFgs (K) \cap \pi_{2,-4}^{-1} \WFgs (u) \right)= \emptyset
\end{equation*}
or equivalently 
\begin{equation*}
\pi_{1,3}^{-1} \overline{\wt \Gamma}_{z_0,2 \ep} \cap \WFgs (K) \cap \pi_{2,-4}^{-1} \WFgs (u) = \emptyset. 
\end{equation*}

Due to assumption \eqref{eq:WFKjempty} we may strengthen this into 
\begin{equation*}
\pi_{1,3}^{-1} \, (\overline{\wt \Gamma}_{z_0,2 \ep} \cup \{ 0 \} ) \setminus 0 \cap \WFgs (K) \cap \pi_{2,-4}^{-1} \, (\WFgs (u) \cup \{ 0 \} ) \setminus 0 = \emptyset.  
\end{equation*}
Note that $\pi_{1,3}^{-1} \, (\overline{\wt \Gamma}_{z_0,2 \ep} \cup \{ 0 \} ) \setminus 0$,  
$\WFgs (K)$,
and $\pi_{2,-4}^{-1} \, (\WFgs (u) \cup \{ 0 \} ) \setminus 0$ are all closed and $s$-conic
subsets of $T^* \rr {2d} \setminus 0$. 

Now \cite[Lemma~5.4]{Wahlberg3} gives the following conclusion. 
There exists open $s$-conic subsets $\Gamma_1 \subseteq T^* \rr {2d} \setminus 0$ 
and $\Gamma_2 \subseteq T^* \rr d \setminus 0$
such that  
\begin{equation*}
\WFgs (K) \subseteq \Gamma_1, \quad \WFgs (u) \subseteq \Gamma_2
\end{equation*}
and 
\begin{equation}\label{eq:emptyintersection}
\pi_{1,3}^{-1} \overline{\wt \Gamma}_{z_0,2 \ep}
\cap \Gamma_1 \cap \pi_{2,-4}^{-1} \Gamma_2 = \emptyset. 
\end{equation}
By intersecting $\Gamma_1$ with the set $\Gamma_1$ defined in \eqref{eq:Gamma1}, 
we may by Lemma \ref{lem:WFkernelaxes} assume that \eqref{eq:Gamma1} holds true.

We will now start to estimate the integral \eqref{eq:opSTFTest1} when $(x,\xi) \in (x_0, \xi_0) + \rB_\ep$ for some $0 < \ep \leqs \frac12$
and $\lambda \geqs 1$. 

We split the domain $\rr {4d}$ of the integral \eqref{eq:opSTFTest1} into three pieces. 
First we integrate over $\rr {4d} \setminus \Gamma_1'$ where we may use 
\eqref{eq:WFKcomplement}. 
Combined with \eqref{eq:STFTtempered} and \eqref{eq:STFTschwartz} this gives if $(x,\xi) \in (x_0, \xi_0) + \rB_\ep$ 
for some $k \in \no$ and any $n,N \in \no$
\begin{equation}\label{eq:opSTFTsubest1}
\begin{aligned}
& \int_{\rr {4d} \setminus \Gamma_1'} | V_\Phi K (y,z,\eta,-\theta) | \, |V_\fy \fy ( \lambda x-y, \lambda^s \xi-\eta) | \, | V_{\overline \fy} u(z,\theta) | \, \dd y \, \dd z \, \dd \eta \, \dd \theta \\
& \lesssim \int_{\rr {4d} \setminus \Gamma_1'} \eabs{(y,z,\eta,\theta)}^{-N} \, \eabs{ (\lambda x-y, \lambda^s \xi-\eta) }^{-n} \, \eabs{(z,\theta) }^k \, \dd y \, \dd z \, \dd \eta \, \dd \theta \\
& \lesssim \eabs{ (\lambda x, \lambda^s \xi) }^{-n} \int_{\rr {4d} \setminus \Gamma_1'} \eabs{(y,z,\eta,\theta)}^{-N} \, \eabs{(y,\eta) }^{n}  \, \eabs{(z,\theta) }^{k} \, \dd y \, \dd z \, \dd \eta \, \dd \theta \\
& \leqs \left( \lambda^{2 \min(1,s)} |(x,\xi)| ^2 \right)^{-\frac{n}{2}} \int_{\rr {4d}} \eabs{(y,z,\eta,\theta)}^{-N+ n + k} \dd y \, \dd z \, \dd \eta \, \dd \theta \\
& \lesssim \lambda^{- n \min(1,s)} 2^{n}
\end{aligned}
\end{equation}
provided $N$ is sufficiently large. 

It remains to estimate the integral \eqref{eq:opSTFTest1} over $(y,z,\eta, - \theta) \in \Gamma_1$ where
we may use \eqref{eq:Gamma1proportionality}. 
By \eqref{eq:emptyintersection} we have 
\begin{equation}\label{eq:Gamma1union}
\Gamma_1 \subseteq \Omega_0 \cup \Omega_2
\end{equation}
where 
\begin{equation*}
\Omega_0 = \Gamma_1 \setminus \pi_{1,3}^{-1} \overline{\wt \Gamma}_{z_0,2 \ep}, \quad
\Omega_2 = \Gamma_1 \setminus \pi_{2,-4}^{-1} \Gamma_2.
\end{equation*}

First we estimate the integral over $(y,z,\eta, - \theta) \in \Omega_2$. 
Then $(z,\theta) \in \rr {2d} \setminus \Gamma_2$ which is a closed $s$-conic set. 
By the compactness of $\sr {2d-1} \setminus \Gamma_2$ and \eqref{eq:WFgs2}
we obtain the estimates
\begin{equation*}
|V_{\overline \fy} u (z,\theta)| 
\lesssim \eabs{(z,\theta)}^{-N}, \quad (z,\theta) \in \rr {2d} \setminus \Gamma_2, \quad \forall N \geqs 0.
\end{equation*}
Together with \eqref{eq:Gamma1proportionality}, \eqref{eq:STFTtempered} and \eqref{eq:STFTschwartz} 
this gives if $(x,\xi) \in (x_0, \xi_0) + \rB_\ep$ 
for some $m \in \no$ and any $n \in \no$
\begin{equation}\label{eq:opSTFTsubest2}
\begin{aligned}
& \int_{\Omega_2'} | V_\Phi K (y,z,\eta,-\theta) | \, |V_\fy \fy ( \lambda x-y, \lambda^s \xi-\eta) | \, | V_{\overline \fy} u(z,\theta) | \, \dd y \, \dd z \, \dd \eta \, \dd \theta \\
& \lesssim \int_{\Omega_2'} \eabs{(y,z,\eta,\theta)}^{m} | \, \eabs{ (\lambda x-y, \lambda^s \xi-\eta) }^{-n} \, | V_{\overline \fy} u(z,\theta) | \, \dd y \, \dd z \, \dd \eta \, \dd \theta \\
& \lesssim \eabs{ (\lambda x, \lambda^s \xi) }^{-n} \int_{\Omega_2'} \eabs{(y,z,\eta,\theta)}^{- 4 d - 1} \, 
\eabs{(y,z,\eta,\theta)}^{m + 4 d + 1}
\eabs{(y,\eta)}^n | V_{\overline \fy} u(z,\theta) | \, \dd y \, \dd z \, \dd \eta \, \dd \theta \\
& \lesssim 
\lambda^{- n \min(1,s)} 2^{n} \\
& \qquad \times  \int_{\Omega_2'} \eabs{(y,z,\eta,\theta)}^{- 4 d - 1} \,
\eabs{(z,\theta)}^{(m + 4 d + 1) \left( 1 + \max \left( s,\frac1s \right) \right) + n \max \left( s,\frac1s \right) } | V_{\overline \fy} u(z,\theta) | \, \dd y \, \dd z \, \dd \eta \, \dd \theta \\
& \lesssim 
\lambda^{- n \min(1,s)} 
\sup_{w \in \rr {2d} \setminus \Gamma_2} \eabs{w}^{ (m + 4 d + 1) \left( 1 + \max \left( s,\frac1s \right) \right) + n \max \left( s,\frac1s \right) } | V_{\overline \fy} u(w) | \\
& \qquad \qquad \qquad \times \int_{\rr {4d}} \eabs{(y,z,\eta,\theta)}^{- 4 d - 1} \dd y \, \dd z \, \dd \eta \, \dd \theta \\
& \lesssim \lambda^{- n \min(1,s)}. 
\end{aligned}
\end{equation}

Finally we need to estimate the integral over $(y,z,\eta, - \theta) \in \Omega_0$. 
Then $(y,\eta) \in \rr {2d} \setminus \overline{\wt \Gamma}_{z_0, 2 \ep}$. 
Hence
\begin{equation*}
\left| z_0 - \left( \lambda^{-1} y, \lambda^{-s} \eta \right) \right| \geqs 2 \ep \quad \forall \lambda > 0 \quad \forall (y,\eta) \in \rr {2d} \setminus \overline{\wt \Gamma}_{z_0, 2 \ep}  
\end{equation*}
and we have for $(x,\xi) \in z_0 + \rB_\ep$
\begin{equation*}
\left| (x,\xi) - \left( \lambda^{-1} y, \lambda^{-s} \eta \right) \right| \geqs \ep \quad \forall \lambda > 0 \quad \forall (y,\eta) \in \rr {2d} \setminus \overline{\wt \Gamma}_{z_0, 2 \ep}. 
\end{equation*}

It follows that for $\lambda \geqs 1$, $(x,\xi) \in z_0 + \rB_\ep$ and 
$(y,\eta) \in \rr {2d} \setminus \overline{\wt \Gamma}_{z_0, 2 \ep}$ we have
\begin{align*}
\left| ( \lambda x, \lambda^s \xi) - (y,\eta) \right|^2 
& = \lambda^2 | x -  \lambda^{-1} y |^2 + \lambda^{2s} | \xi - \lambda^{-s} \eta |^2 \\
& \geqs \lambda^{2 \min (1,s)} \ep^2. 
\end{align*}

Together with \eqref{eq:Gamma1proportionality}, \eqref{eq:STFTtempered} and \eqref{eq:STFTschwartz} 
this gives if $(x,\xi) \in (x_0, \xi_0) + \rB_\ep$ 
for some $m,k \in \no$ and any $n,N \in \no$
\begin{equation}\label{eq:opSTFTsubest3}
\begin{aligned}
& \int_{\Omega_0'} | V_\Phi K (y,z,\eta,-\theta) | \, |V_\fy \fy ( \lambda x-y, \lambda^s \xi-\eta) | \, | V_{\overline \fy} u(z,\theta) | \, \dd y \, \dd z \, \dd \eta \, \dd \theta \\
& \lesssim \int_{\Omega_0'}  \eabs{(y,z,\eta,\theta)}^{m} \, \eabs{ (\lambda x-y, \lambda^s \xi-\eta) }^{-n-N} \, 
\eabs{(z,\theta)}^k
\, \dd y \, \dd z \, \dd \eta \, \dd \theta \\
& \lesssim 
\lambda^{- n \min(1,s)}
\int_{\Omega_0'} \eabs{(y,z,\eta,\theta)}^{- 4 d - 1} \, 
\eabs{ (\lambda x-y, \lambda^s \xi-\eta) }^{-N} \\
& \qquad \qquad \qquad \qquad \qquad \qquad \qquad \qquad \times
\eabs{(z,\theta)}^{ k + (m + 4 d + 1) \left( 1 + \max \left( s,\frac1s \right) \right) }
\, \dd y \, \dd z \, \dd \eta \, \dd \theta \\
& \lesssim 
\lambda^{- n \min(1,s)}
\eabs{ (\lambda x, \lambda^s \xi) }^{N}
\int_{\Omega_0'} \eabs{(y,z,\eta,\theta)}^{- 4 d - 1} \, \eabs{ (y, \eta) }^{-N} \\
& \qquad \qquad \qquad \qquad \qquad \qquad \qquad \qquad \times
\eabs{(z,\theta)}^{ k + (m + 4 d + 1) \left( 1 + \max \left( s,\frac1s \right) \right) }
\, \dd y \, \dd z \, \dd \eta \, \dd \theta \\
& \lesssim 
C_N \lambda^{- n \min(1,s) + N \max(1,s) } \\
& \qquad \times 
\int_{\rr {4d}} \eabs{(y,z,\eta,\theta)}^{- 4 d - 1} \,
\eabs{(z,\theta)}^{- N \min \left( s,\frac1s \right) + k + (m + 4 d + 1) \left( 1 + \max \left( s,\frac1s \right) \right)}
\, \dd y \, \dd z \, \dd \eta \, \dd \theta \\
& \lesssim 
C_N \lambda^{- n \min(1,s) + N \max(1,s) }
\end{aligned}
\end{equation}
if $N$ is large enough. 

Combining \eqref{eq:opSTFTsubest1}, \eqref{eq:opSTFTsubest2} and \eqref{eq:opSTFTsubest3}
and taking into account \eqref{eq:Gamma1union}, we have by \eqref{eq:opSTFTest1} shown
\begin{equation*}
\sup_{(x,\xi) \in (x_0, \xi_0) + \rB_\ep, \ \lambda > 0} \lambda^n | V_\fy(\cK u) ( \lambda x, \lambda^s \xi) | 
< + \infty \quad \forall n \geqs 0
\end{equation*}
which finally proves the claim $z_0 \notin \WF_{\rm g}^s (\cK u)$. 
\end{proof}

\section{Propagation of the $s$-Gabor wave front set for certain evolution equations}\label{sec:schrodinger}

In \cite[Remark~4.7]{Rodino3} we discuss the initial value Cauchy problem
for the evolution equation 
in dimension $d = 1$ 
\begin{equation}\label{eq:schrodeq2}
\left\{
\begin{array}{rl}
\partial_t u(t,x) + i D_x^{m} u (t,x) & = 0, \quad m \in \no \setminus 0, \quad x \in \ro, \quad t \in \ro, \\
u(0,\cdot) & = u_0. 
\end{array}
\right.
\end{equation}
It is a generalization of the Schr\"odinger equation for the free particle where $m = 2$. 

Here we generalize this equation into
\begin{equation}\label{eq:schrodeq3}
\left\{
\begin{array}{rl}
\partial_t u(t,x) + i p(D_x) u (t,x) & = 0, \quad x \in \rr d, \quad t \in \ro, \\
u(0,\cdot) & = u_0 
\end{array}
\right.
\end{equation}
where $p: \rr d \to \ro$ is a polynomial with real coefficients of order $m \geqs 2$, that is
\begin{equation}\label{eq:polynomial1}
p (\xi) = \sum_{|\alpha| \leqs m} c_\alpha \xi^\alpha, \quad c_\alpha \in \ro.  
\end{equation}
The principal part is 
\begin{equation}\label{eq:principalpart1}
p_m (\xi) = \sum_{|\alpha| = m} c_\alpha \xi^\alpha  
\end{equation}
and there exists $\alpha \in \nn d$ such that $|\alpha| = m$ and $c_\alpha \neq 0$.

The Hamiltonian is $p(\xi)$, and the Hamiltonian flow of the principal part $p_m(\xi)$ is given by
\begin{equation}\label{eq:hamiltonflow}
( x (t),\xi(t) ) = \chi_t (x, \xi)
= (x + t \nabla p_m (\xi), \xi), \quad t \in \ro, \quad (x, \xi) \in T^* \rr d \setminus 0.
\end{equation}

The explicit solution to \eqref{eq:schrodeq3} is 
\begin{equation}\label{eq:schrodingerpropagator}
u (t,x) 
= e^{- i t p(D_x)} u_0 
= (2 \pi)^{- \frac{d}{2}} \int_{\rr d} e^{i \la x, \xi \ra - i t p(\xi)} \wh u_0 (\xi) \dd \xi
\end{equation}
for $u_0 \in \cS(\rr d)$. 
Thus $u (t,x) = \cK_t u_0(x)$ where $\cK_t$ is the operator with Schwartz kernel
\begin{align*}
K_t (x,y) & = (2 \pi)^{-d} \int_{\rr d} e^{i \la x-y, \xi \ra - i t p(\xi)} \dd \xi \\
& =  (2 \pi)^{- \frac{d}{2}} \cF^{-1} (e^{- i t p}) (x-y). 
\end{align*}
The propagator $\cK_t$ is a convolution operator with kernel
\begin{equation}\label{eq:convolutionkernel}
k_t  =  (2 \pi)^{- \frac{d}{2}} \cF^{-1} (e^{- i t p}) \in \cS'(\rr d)
\end{equation}
and we may write
\begin{equation}\label{eq:schwartzkernel1}
K_t (x,y) 
=  \left( 1 \otimes k_t \right) \circ \kappa^{-1}(x,y) \in \cS'(\rr {2d})
\end{equation}
where $\kappa \in \rr {2d \times {2d}}$ is the matrix defined by $\kappa(x,y) = (x+\frac{y}{2}, x - \frac{y}{2})$ for $x,y \in \rr d$. 

It follows from \eqref{eq:schrodingerpropagator} that $\cK_t: \cS(\rr d) \to \cS(\rr d)$ is continuous, 
invertible with inverse $\cK_{-t}$, and $\cK_{-t} = \cK_t^*$ which denotes the adjoint. 
Defining for $u \in \cS'(\rr d)$
\begin{equation*}
(\cK_t u, \psi) = (u, \cK_t^* \psi), \quad \psi \in \cS(\rr d), 
\end{equation*}
gives a unique continuous extension $\cK_t: \cS'(\rr d) \to \cS'(\rr d)$.

As we will see now the continuity of $\cK_t$ on $\cS(\rr d)$ and the unique extension to a continuous operator on $\cS'(\rr d)$
may alternatively be proved as a consequence of 
$\WF_{\rm g,1}^s(K_t) = \WF_{\rm g,2}^s(K_t) = \emptyset$
and Proposition \ref{prop:opSTFTformula}. 

The next result shows that $\cK_t$ propagates the anisotropic $s$-Gabor wave front set along the Hamiltonian flow of $p_m$
if $s = \frac{1}{m-1}$,
whereas the anisotropic $s$-Gabor wave front set is invariant if $s < \frac{1}{m-1}$.
In the proof we use the symplectic matrix 
\begin{equation*}
\J =
\left(
\begin{array}{cc}
0 & I_d \\
-I_d & 0
\end{array}
\right) \in \rr {2d \times 2d}. 
\end{equation*}

\begin{thm}\label{thm:WFgskernel}
Let $m \geqs 2$ and let $p$ be defined 
by \eqref{eq:polynomial1}, \eqref{eq:principalpart1}, 
and denote by \eqref{eq:hamiltonflow} the Hamiltonian flow
of the principal part $p_m$. 
Suppose $\cK_t : \cS(\rr d) \to \cS(\rr d)$ is the continuous linear operator with Schwartz kernel 
\eqref{eq:schwartzkernel1} where $k_t$ is defined by \eqref{eq:convolutionkernel}. 
Then if $0 < s \leqs \frac{1}{m-1}$ we have for $u \in \cS' (\rr d)$ and $t \in \ro$

\begin{align}
& \WFgs   ( \cK_t u) = \chi_t  \left( \WFgs (u) \right), \quad s = \frac{1}{m-1}, \label{eq:propagation1} \\
& \WFgs  ( \cK_t u) = \WFgs  (u), \quad s < \frac{1}{m-1}. \label{eq:propagation2} 
\end{align}
\end{thm}

\begin{proof}
First let $s = \frac{1}{m-1}$. 
By \cite[Theorem~7.1]{Rodino4} we have 
\begin{equation*}
\WF_{\rm g}^{m-1} ( e^{- i t p} ) 
\subseteq \{ (x, - t \nabla p_m(x) ) \in T^* \rr d: \ x \neq 0 \}, 
\end{equation*}
and from \cite[Eq.~(4.6) and Proposition~4.3 (i)]{Rodino4} we obtain
\begin{align*}
\WF_{\rm g}^{s} ( k_t ) 
& = \WF_{\rm g}^{s} ( \cF^{-1} e^{- i t p} ) 
= - \WF_{\rm g}^{s} ( \cF e^{- i t p} ) \\
& = - \J \WF_{\rm g}^{m-1} ( e^{- i t p} ) \\
& \subseteq \{ ( t \nabla p_m(x)  , x) \in T^* \rr d: \ x \neq 0 \}. 
\end{align*}

Now  \eqref{eq:schwartzkernel1}, \cite[Proposition~4.3~(ii)]{Rodino4}, Proposition \ref{prop:tensorWFs} and \cite[Proposition~5.3 (iii)]{Rodino4} yield
\begin{align*}
& \WFgs (K_t)
= \WFgs ( \left( 1 \otimes k_t \right) \circ \kappa^{-1} ) \\
& = 
\left( 
\begin{array}{cc}
  \kappa & 0 \\
  0 & \kappa^{-T}
  \end{array}
\right) 
\WFgs \left( 1 \otimes k_t \right) \\
& \subseteq \{ ( \kappa( x_1, x_2), \kappa^{-T}(\xi_1, \xi_2) ) \in T^* \rr {2d}: \\
& \qquad \qquad \qquad \qquad (x_1, \xi_1) \in \WFgs (1) \cup \{ 0 \},  
(x_2, \xi_2) \in \WFgs (k_t) \cup \{ 0 \} \} \setminus 0 \\
& = \{ ( \kappa( x_1, t \nabla p_m(x_2) ), \kappa^{-T} (0, x_2) ) \in T^* \rr {2d}: \ x_1, x_2 \in \rr d \} \setminus 0 \\
& = \left\{ \left( x_1 + t \frac12 \nabla p_m(x_2), x_1 - t \frac12 \nabla p_m(x_2), x_2, - x_2 \right) \in T^* \rr {2d}: \ x_1, x_2 \in \rr d  \right\} \setminus 0 \\
& = \left\{ \left( x_1 + t \nabla p_m(x_2), x_1, x_2, - x_2 \right) \in T^* \rr {2d}: \ x_1, x_2 \in \rr d \right\} \setminus 0. 
\end{align*}

Since $m \geqs 2$ we have $\nabla p_m(0) = 0$ and 
$\WF_{\rm g,1}^s(K_t) = \WF_{\rm g,2}^s(K_t) = \emptyset$ follows. 
By Proposition \ref{prop:opSTFTformula} we thus obtain an alternative proof of the already known fact that 
$\cK_t$ is continuous 
on $\cS(\rr d)$ and extends uniquely to be continuous 
on $\cS'(\rr d)$. 
Invertibility also follows since $\cK_t^{-1} = \cK_{-t}$. 
Moreover we may apply Theorem \ref{thm:WFgspropagation} which gives for $u \in \cS'(\rr d)$
\begin{align*}
\WFgs( \cK_t u) 
& \subseteq \WFgs (K_t)' \circ \WFgs (u) \\
& = \{ (x,\xi) \in T^* \rr d: \ \exists (y,\eta) \in \WFgs (u), \ (x,y, \xi, - \eta) \in \WFgs (K_t) \} \\
& \subseteq \{ ( x_1 + t \nabla p_m(x_2) , x_2) \in T^* \rr d: \ (x_1,x_2) \in \WFgs (u) \} \\
& = \chi_t  \left( \WFgs (u) \right).
\end{align*}

The opposite inclusion follows from $\cK_{t}^{-1} = \cK_{-t}$, 
\begin{equation*}
\WFgs( u) 
= \WFgs( \cK_{-t}  \cK_t u) \\
\subseteq \chi_{-t}  \left( \WFgs ( \cK_t u) \right)
\end{equation*}
and $\chi_{-t} = \chi_t^{-1}$. 
We have proved \eqref{eq:propagation1}.

It remains to consider the case $s < \frac1{m-1}$. 
By \cite[Theorem~7.2]{Rodino4} we have 
\begin{equation*}
\WFg^{\frac1{s}} ( e^{- i t p} ) 
\subseteq ( \rr d \setminus 0 ) \times \{ 0 \}
\end{equation*}
and from \cite[Eq.~(4.6) and Proposition~4.3 (i)]{Rodino4} we obtain
\begin{equation*}
\WFgs ( k_t ) 
= - \J \WFg^{\frac1{s}} ( e^{- i t p} ) \\
\subseteq \{ 0 \} \times ( \rr d \setminus 0 ).  
\end{equation*}

Again \eqref{eq:schwartzkernel1}, \cite[Propositions~4.3 (ii) and 5.3 (iii)]{Rodino4}, 
and Proposition \ref{prop:tensorWFs} yield
\begin{align*}
\WFgs (K_t)
& \subseteq \{ ( \kappa( x_1, x_2), \kappa^{-T}(\xi_1, \xi_2) ) \in T^* \rr {2d}: \\
& \qquad \qquad \qquad \qquad (x_1, \xi_1) \in \WFgs (1) \cup \{ 0 \},  
\ (x_2, \xi_2) \in \WFgs (k_t) \cup \{ 0 \} \} \setminus 0 \\
& \subseteq \{ ( \kappa( x_1, 0 ), \kappa^{-T}(0, x_2) \in T^* \rr {2d}: \ x_1, x_2 \in \rr d \} \setminus 0 \\
& = \left\{ \left( x_1 , x_1, x_2, - x_2 \right) \in T^* \rr {2d}: \ x_1, x_2 \in \rr d \right\} \setminus 0. 
\end{align*}

Again we have $\WF_{\rm g,1}^s(K_t) = \WF_{\rm g,2}^s(K_t) = \emptyset$, 
and Proposition \ref{prop:opSTFTformula} 
gives continuity on $\cS (\rr d)$ and on $\cS' (\rr d)$; 
the invertibility also follows since $\cK_t^{-1} = \cK_{-t}$. 
Now Theorem \ref{thm:WFgspropagation} gives for $u \in \cS'(\rr d)$
\begin{equation*}
\WFgs ( \cK_t u) 
\subseteq \WFgs (K_t)' \circ \WFgs (u) 
\subseteq \WFgs (u).
\end{equation*}

The opposite inclusion again follows from $\cK_{t}^{-1} = \cK_{-t}$ and $\chi_{-t} = \chi_t^{-1}$. 
We have proved \eqref{eq:propagation2}.
\end{proof}

\begin{rem}\label{rem:continuityWFK2}
If $s > \frac1{m-1}$ and $p_m (x) \neq 0$ for all $x \in \rr d \setminus 0$
then by \cite[Theorem~7.3]{Rodino4} 
\begin{equation*}
\WFg^{\frac1{s}} ( e^{- i t p} ) 
\subseteq \{ 0 \} \times ( \rr d \setminus 0 ) 
\end{equation*}
so \cite[Eq.~(4.6) and Proposition~4.3 (i)]{Rodino4} give
\begin{equation*}
\WFgs ( k_t ) 
= - \J \WFg^{\frac1{s}} ( e^{- i t p} ) \\
\subseteq ( \rr d \setminus 0 ) \times \{ 0 \}.  
\end{equation*}

Again \eqref{eq:schwartzkernel1}, \cite[Propositions~4.3 (ii) and 5.3 (iii)]{Rodino4}, 
and Proposition \ref{prop:tensorWFs} yield
\begin{align*}
& \WFgs (K_t)
\subseteq \{ ( \kappa( x_1, x_2), \kappa^{-T}(\xi_1, \xi_2) ) \in T^* \rr {2d}: \\
& \qquad \qquad \qquad \qquad (x_1, \xi_1) \in \WFgs (1) \cup \{ 0 \},  
\ (x_2, \xi_2) \in \WFgs (k_t) \cup \{ 0 \} \} \setminus 0 \\
& \subseteq \{ ( \kappa( x_1, x_2 ), \kappa^{-T}(0,0) \in T^* \rr {2d}: \ x_1, x_2 \in \rr d \} \setminus 0 \\
& = ( \rr {2d} \setminus 0 ) \times \{ 0 \}. 
\end{align*}
In this case we cannot conclude that $\WF_{\rm g,1}^s(K_t)$ and $\WF_{\rm g,2}^s(K_t)$ are empty.

Thus we cannot conclude any statement on propagation of the anisotropic $s$-Gabor wave front set 
from Theorem \ref{thm:WFgspropagation} when $s > \frac1{m-1}$. 
\end{rem}

\section*{Acknowledgment}
Work partially supported by the MIUR project ``Dipartimenti di Eccellenza 2018-2022'' (CUP E11G18000350001).


\end{document}